\documentclass[preprint,11pt]{revtex4-1}

\usepackage{amsmath}  

\usepackage{amsthm} 
\usepackage{amsfonts}  
\usepackage{graphicx}   
\usepackage{verbatim}   
\usepackage{color}      
\usepackage{subfigure}  
\usepackage[displaymath, mathlines]{lineno}
\numberwithin{equation}{section}
\usepackage[titletoc,title]{appendix}
\usepackage{mathtools}
\usepackage{dsfont}
\newtheorem{defi}{{\it Definition}}[section]
\newtheorem{coro}{{\it Corollary}}[section]

\newtheorem{theorem}{{\it Theorem}}[section]

\newtheorem{prop}{{\it Proposition}}[section]
\newcommand{\R}{\mathbb{R}}
\newcommand{\pot}{\gamma}
\renewcommand{\P}{\mathbb{P}}
\newcommand{\KeyCond}{\sup_{s>0}\int_{\R}\denx(v)\frac{t}{(v-s)^2+t^2}dv\leq {a_x}}
\newcommand{\KeyCondC}{\int_{\R}\denx(v)\frac{1}{(v-s)^2+t^2}dv&\leq& {||\denx||_\infty}\int_{\R}\frac{t}{(v-s)^2+t^2}dv}
\newcommand{\KeyInt}{\int_{\R}\denx(v)\frac{1}{(v-s)^2+t^2}dv}

\newcommand{\COEF}{\akv}
\newcommand{\LArk}{(1-\lambda_{r_k})}

\newcommand{\avmu}{\mu^{av}}

\newcommand{\SG}{{G}}
\newcommand{\trB}{\tr[\mathds{1}_{B_k}f(\Hw)]}
\newcommand{\trzB}{\tr[\mathds{1}_{B_k}f_z(\Hw)]}
\newcommand{\EtrB}{E[\tr[\mathds{1}_{B_k}f(\Hw)]]}

\newcommand{\PIx}{\langle \delta_{x},P_I(\Hwk) \delta_x\rangle}
\newcommand{\RIx}{\left \langle \delta_{x},\frac{1}{\Hwk-s-it} \delta_x\right \rangle}
\newcommand{\trzBk}{\tr[\mathds{1}_{B_k}f_z(\Hwk)]}
\newcommand{\trBrk}{\tr[\mathds{1}_{B_k}f(\Hwrk)]}
\newcommand{\trzBrk}{\tr[\mathds{1}_{B_k}f_z(\Hwrk)]}
\newcommand{\xikj}{\xi_{k,j}}
\newcommand{\Dkw}[1]{D_{k,\omega}^{(#1)}}
\newcommand{\xik}{\xi_{k}}
\newcommand{\txik}{\tilde{\xi}_{k}}
\newcommand{\xikjr}{\xi_{k,j}}
\newcommand{\txikj}{\tilde{\xi}_{k,j}}
\newcommand{\txikjb}{\tilde{\xi}_{k,j^\prime }}
\newcommand{\Hw}{H^{\omega}}

\newcommand{\Vw}{V^{\omega}}
\newcommand{\Vwk}{V_{B_{k}}}
\newcommand{\Hwk}{\Hw_k}

\newcommand{\Hwrk}{\Hw_{r_k}}
\newcommand{\tHwrk}{\tilde{H}^\omega_{r_k}(0,0)}

\newcommand{\X}{\mathbb{X}}
\newcommand{\N}{\mathbb{N}}
\newcommand{\Pk}{\P_k}
\newcommand{\Etx}[1]{E[\langle\delta_#1,f(\Hw)\delta_#1 \rangle]}
\newcommand{\denx}{\rho_x}
\newcommand{\deny}{\rho_y}

\newcommand{\EHxkz}[1]{E[\im \langle\delta_#1,(\Hw_{r_k}-z_k)^{-1}\delta_#1 \rangle]}
\newcommand{\Hxkz}[1]{\im \langle\delta_#1,(\Hw_{r_k}-z_k)^{-1}\delta_#1 \rangle}

\newcommand{\EHxxk}[1]{E[\im \langle\delta_#1,(\Hw_k-z_k)^{-1}\delta_#1 \rangle]}

\newcommand{\EHxbz}[1]{E[\im \langle\delta_#1,(\Hw-z_k)^{-1}\delta_#1 \rangle]}

\newcommand{\EHxxbk}[1]{E[\im \langle\delta_#1,(\Hw_{r_k}-z_k)^{-1}\delta_#1 \rangle]}
\DeclareMathOperator{\im}{Im}
\DeclareMathOperator{\tr}{Tr}
\DeclareMathOperator{\re}{Re}
\DeclareMathOperator{\Var}{Var}

\newcommand{\tx}[1]{\langle\delta_#1,f(\Hw)\delta_#1 \rangle}
\newcommand{\txo}[1]{\langle\delta_#1,f(\Hw_{#1})\delta_#1 \rangle}
\newcommand{\ak}{A_k}
\newcommand{\ax}{a_x}

\newcommand{\akv}{A_k}
\newcommand{\ake}{A_k}

\newcommand{\akj}{A_{k,j}}
\newcommand{\akesq}{A_{k}^2 }
\newcommand{\akjesq}{A_{k,j}^2 }
\newcommand{\akjv}{A_{k,j}}
\newcommand{\akjvsq}{A_{k,j}^2 }
\newcommand{\akvsq}{A_{k}^2 }

\newcommand{\pxA}{p_{x}(\BS)}

\newcommand{\BS}{I}
\newcommand{\Uxo}{U_x}
\newcommand{\Uyo}{U_y}
\newcommand{\ewk}[1]{e_#1^{\omega,k}}

\newcommand{\Zkj}{W_{k,j}(z)}
\newcommand{\LEBM}{|I|}
\newcommand{\Zkjzk}{W_{k,j}(z_k)}

\newcommand{\UHz}[1]{\langle\delta_#1,(\Hwrk-z)^{-1}\delta_#1\rangle -E[\langle\delta_#1,(\Hwrk-z)^{-1}\delta_#1 \rangle]}

\newcommand\lra{{\rightarrow}}

\newcommand\lmt{{\longmapsto}}

\newcommand{\cH}{\mathcal{H}}

\newcommand{\cF}{\mathcal{F}}

\newcommand{\cB}{\mathcal{B}}

\newcommand{\cD}{\mathcal{D}}

\newcommand{\C}{\mathbb{C}}

\DeclarePairedDelimiter\floor{\lfloor}{\rfloor}
\begin{document}

\title{Poisson Statistics in the Non-Homogeneous Hierarchical Anderson Model }
\author{Jorge  Littin}
\email{jlittin@ucn.cl}
\altaffiliation{}
\affiliation{Universidad Cat\'olica del Norte, Departamento de Matem\'aticas,\phantom{XXXXXXXX} Angamos 0610, Antofagasta Chile.\phantom{XXXXXXXX}}        

\date{\today}

\begin{abstract}
In this article we study the problem of localization of eigenvalues for the non-homogeneous hierarchical Anderson model. More specifically, given the hierarchical Anderson model with spectral dimension $0<d<1$ with a random potential acting on the diagonal of non i.i.d. random variables,  sufficient conditions on the disorder are provided in order to obtain  the two main results: the weak convergence of the counting measure for almost all  realization of the random potential and the weak convergence of the re-scaled eigenvalue counting measure to a Poisson point process. The technical part improves the already existing arguments of Kritchevski \cite{Krit1,Krit2}, who studied the hierarchical model with a disorder acting on the diagonal, with independent and identically distributed random variables, by using the  argument of Minami \cite{Minami1}. At the end of this article, we study an application example that allows us to understand some relations between the spectral dimension of the hierarchical Laplacian and the magnitude of the disorder.

\end{abstract}
\maketitle
\section{Introduction}
The  Schr\"odinger's discrete operator, also known as the Anderson's model, can be described in the form
\begin{equation}
\Hw_{k}=H_0+\Vw
\end{equation}
where $H_0$ is a self-adjoint operator defined in an appropriate Hilbert space and $\Vw$ is a random operator acting on the diagonal, whose components are assumed independent and  identically distributed. In particular, when  $H_0$ is the discrete Laplacian on $ \mathbb{Z}^d $ with  nearest neighbor interactions
\begin{equation}\label{classic}
(H_0\psi)(x)=\sum_{|y-x|=1}\psi(y)
\end{equation}
is called  the Anderson tight binding model. Minami \cite{Minami1}, showed for this model that the eigenvalues counting measure  converges to a Poisson point process when a correct rescaling is applied. From a technical point of view,  he used the technique of successive approximations of finite-range operators. This argument was subsequently used in different random  models \cite{Aizenman2006,Bellissard2007,killip2009,Killip2007} and it is considered as a ''standard argument" to prove the convergence to a Poisson point process.  An important case, which mainly motivates the  realization of this article, is the Hierarchical Anderson Model 
\begin{equation}
H_0=\sum_{r=0}^{\infty}p_{r}E_{r}
\end{equation}
where $p_0=0$, $p_r \geq 0$ $r\geq 1$ is a sequence of non negative numbers satisfying $\sum_{r \geq 0}p_r=1$ and $E_r$ is a sequence  of operators that we will define more precisely afterwards.  Kritchevski \cite{Krit1,Krit2} studied the random model, where the potential is a sequence of independent random variables identically distributed. In particular, he proved the weak convergence of the counting measure of eigenvalues and the convergence to a Poisson point process for a suitable rescaled counting measure. Subsequently, Combes, Germinet \& Klein \cite{Combes2009} obtained generalized eigenvalue counting estimates.

In this article we prove the convergence to a Poisson point process for the Hierarchical Anderson Model with a random potential acting on the diagonal with independent but not necessarily  identically distributed random variables. This is a generalization of the result previously given in \cite{Krit1}. In particular, we provide sufficient conditions that relate the spectral dimension of the hierarchical Laplacian and the magnitude of the randomness in order to obtain a limiting Poisson point process for the rescaled counting measure.

The article is organized as follows: in section \ref{Desc} we present some preliminaries and an introductory example, in order to submit the main Hypothesis of this article. In section \ref{DOS}
we prove the weak convergence of the spectral counting measure for almost all realization of the random potential; in section \ref{POISSTA} we prove the weak convergence to a Poisson point process for a convenient rescaled counting measure. Finally, in section \ref{EXAMPLES} we study some examples, which are the original motivation for this article.

					\section{Model Description}\label{Desc}
The discrete hierarchical Laplacian is a well known self adjoint operator and can be defined as follows: given the countable set $\X=\{0,1,2,\cdots\}$ and $n \geq 2$, the hierarchical distance is defined as 
\begin{equation}\label{Dysdist}
d(x,y)=\min \{k \in \N_0\mid q(x,n^r)=q(y,n^k)\}
\end{equation}
where $q(x,n^k)$ denotes the quotient of the division of $x$ by $n^k$. The closed ball centered at point $x$ with radius $k$ is denoted by
$$B(x,k)=\{ y \in \X \mid d(x,y)\leq k\}.$$
The fundamental property of the hierarchical distance is that two balls with the same radius are either disjoint or identical and that each ball $B(x,r+1)$ is the disjoint union of $n$ balls of radius $r$.   In addition, we consider the Hilbert space $l^{2}(\X)$ of complex valued functions $\phi:\X\lmt\mathbb{C}$ satisfying
\begin{equation}
\sum_{x\in\X}|\phi(x)|^{2}<\infty,
\end{equation}
 the inner product on $l^2(\X)$
$$\left\langle\phi,\varphi \right\rangle:=\sum_{x\in\X}\overline{\phi(x)}\varphi(x)$$
and the following family of operators: given $r\geq 0$, we define  $E_{r}$ as
$$(E_{r}\phi)(x):=\frac{1}{|B(x,r)|}\sum_{y\in B(x,r)}\phi(y).$$
The operator $E_{r}$ is the orthogonal projection over the closed subspace $\cH_{r}\subset l^{2}(\X)$, where $\cH_{r}$ is the finite dimensional subspace of $\ell^2(\X)$ consisting of all those functions $\phi\in l^{2}(\X)$ taking a constant value on $B(x,r)$. We give the formal definition of the Hierarchical Laplacian below.

 \begin{defi}
 Given the sequence of operators $E_r$, $r \geq 0$ and a sequence of non-negative numbers $p_r$, $r \geq 0$ satisfying $p_0=0$ and $\sum_{r=1}^{\infty} p_{r}=1$, the hierarchical Laplacian is defined as 
\begin{equation}
\Delta:=\sum_{r=0}^{\infty}p_{r}E_{r}.
\end{equation}
\end{defi}
It is a well known fact (see for instance \cite{Krit1}, \cite{Mol2}) that the hierarchical Laplacian is a bounded self adjoint operator with a sequence of infinitely degenerated isolated eigenvalues  
\begin{equation}
\lambda_{k}:=\sum_{s=0}^{k} p_{s},\quad r \geq 0
\end{equation}
 being $\lambda_\infty=1$ an accumulation point but not an eigenvalue. Connected to the spectrum, the spectral dimension is defined as
 	\begin{equation}\label{SpecDim}
 	\frac{2}{d}=-\frac{1}{\ln n}\lim_{k\rightarrow \infty} \frac{\ln (1-\lambda_k)}{k }
 	\end{equation}
 	provided the limit exists.
 We introduce next our random operator: given the countable set $\X$, the hierarchical Laplacian $\Delta$ acting on the space $\ell^2(\X)$ and a random potential $V$ acting on the diagonal we set
\begin{equation}
\Hw=\Delta+\Vw
\end{equation}
here the random operator  $\Vw$ is defined as 
\begin{equation}
(\Vw\psi)(x)=V(x) \psi(x)\hspace{.5cm}\psi \in \ell^2(\X).
\end{equation}
The random variables $V(x)$, denoted by $V_x$ in the following, are defined over a suitable probability  space $(\Omega,\cF,\P)$ are assumed independent but not necessarily identically distributed. 

We emphasize that in the particular case that $V_x$ are i.i.d., we call it the Hierarchical Anderson Model, which was previously introduced in \cite{Krit1,Krit2,Mol1,Mol2}. In this article, we will suppose that the random variable $V_x$, $x \in \X$  are independent with  a continuous, strictly positive and bounded density $\denx$. It means that for any $A \in \mathcal{B}(\R)$.
\begin{equation}
\P[V_x \in A]=\int_A \denx(v)dv
\end{equation}
where $\denx(v)>0$ for all $v \in \R$. In addition we assume that the density is bounded, i.e.
\begin{equation}\label{densup}
||\denx||_\infty=\sup_{v \in \R}|\denx(v)|<\infty
\end{equation}
We remark that  given any realization of the random variables $V_x$, $x \in \mathbb{X}$  the operators $\Vw$ and $\Hw$  are both  unbounded, self-adjoint with domain
\begin{equation}
\cD=\left\{\psi \in \ell^2(\X) \mid  \sum_{x\in\X}|\psi(x)|^{2}(1+V_x^2)<\infty \right\}.
\end{equation}

\subsection{The pure random case}\label{poissson}
As an illustrative example we study first the operator $\Hw =\Vw$, which means that only the random component of the operator is present.  By restricting $\Vw$ to the closed ball $B_k$ for some  $k \geq 0$ fixed, we get directly  that the set of isolated eigenvalues are  $e_{x}^{\omega,k}=V_x$, $1\leq x \leq |B_k|$. So, it seems natural to consider the eigenvalue counting  measure
		\begin{equation}\label{CONTEO}
				\mu_{k}^{\omega}=\frac{1}{\ak}\sum_{x=1}^{|B_{k}|}\delta(e_{x}^{\omega,k})
				\end{equation}
where $\ak$ is a set of non negative numbers satisfying $\ak \rightarrow \infty$ (we will  specify more precisely this sequence  later). We emphasize that $\mu_{k}^{\omega}$ is a random measure. By noticing that $E[\delta(e_{x}^{\omega,k})(\BS)]=\pxA$, 
where $\pxA=P[V_x \in \BS]$, we get that  the expected  counting measure is
 		\begin{equation}\label{ESPE}
				E[\mu_{k}^{\omega}(\BS)]=\frac{1}{\ak}\sum_{x=1}^{|B_{k}|}\pxA.
				\end{equation}
Similarly, we have  $\Var[\delta(e_{x}^{\omega,k})(\BS)]=\pxA(1-\pxA)$. Therefore,  the variance of the counting measure can be calculated explicitly 
 		\begin{equation}\label{VARI}
\Var[\mu_{k}^{\omega}(\BS)]=\frac{1}{\akvsq}\sum_{x=1}^{|B_{k}|}\pxA(1-\pxA).
\end{equation}
Now, from equations \ref{ESPE} and \ref{VARI} it is not hard to check that given a fixed measurable set $\BS\in \cB(\R)$, the sequence of random variables $\delta(e_{x}^{\omega,k})(\BS)$ satisfies the Strong Law of  Large Numbers if $ \limsup_{k \rightarrow \infty} \frac{1}{\ak}\sum_{x=1}^{|B_{k}|}\pxA<\infty$
 (see for instance Theorem 7.5 of Sinai \cite{Sinai1}). Therefore, the limit
\begin{equation}\label{alsure}
\lim_{k\lra \infty}\mu_{k}^{\omega}(\BS)=E[\mu_{k}^{\omega}(\BS)]
\end{equation}
 exists  and its value is finite with probability one. Furthermore , if the following limit exists
\begin{equation}
\lim_{k\lra \infty}\frac{1}{\akv}\sum_{x \in B_k}\denx(v)= F(v)\label{Inten0}
\end{equation}
for some locally integrable function $F(v)$, we can deduce from the dominated 
convergence theorem that 
\begin{equation}\label{alsure2}
\lim_{k\lra \infty}E[\mu_{k}^{\omega}(\BS)]=\int_\BS F(v)dv
\end{equation}
and by consequence $\lim_{k\lra \infty}\mu_{k}^{\omega}(\BS)=\int_\BS F(v)dv$ with probability one. We remark that the  considerations above allow us to prove the existence of a limit for the counting measure \ref{CONTEO},  but it does not provide  additional information about typical number of  eigenvalues, for example in the interval $\left[e-\frac{1}{2\akv},e+\frac{1}{2\akv}\right]$, $e \in \R$. In order to get more precise estimates, we take the following rescaled  counting measure 
\begin{equation}
\xik^{\omega,e}=\sum_{x=1}^{|B_{k}|}\delta(\ake(e_{x}^{\omega,k}-e)).\label{ReScaled}
\end{equation} 
Similarly, we define
	\begin{equation}\label{RESCALEDkj}
	\xikjr^{\omega}= \sum_{x \in B_{k,j}}\delta(\ake(e_{x}^{\omega,k}-e))
\end{equation}
where \ref{RESCALEDkj}  is the re-scaled eigenvalue counting measure  of the operator $\Vw$ restricted to the  subset
\begin{equation}
B_{k,j}=\{jn^r-1,jn^r,\cdots,(j+1)n^r-1\}, \hspace{1cm} 2\leq j \leq n^{k-r}.
\end{equation}
For  $j=1$ we write
\begin{equation}\label{Bkdef}
B_{k,1}=B_{k-1}:=\{0,1,\cdots,n^{k-1}-1\}.
\end{equation}
It is clear that  ${\xi}_{k}^{\omega}=\sum_{j=1}^{n^{k-r}}\xikjr^{\omega}$. Related to the re-scaled counting measure, we introduce the nest definition.
\begin{defi}
	We say that the sequence of non negative numbers $\ax$, $x \in \X$ satisfy the Hypothesis (H) if  there is a sub sequence $r_k$, $k \geq 1$ such that 
		\begin{eqnarray}
		\lim_{k \rightarrow \infty}\sup_{1\leq j\leq n^{k-r_k}}\frac{\akjv}{\akv}&=&0\label{rkchoose}\\
		\lim_{k \rightarrow \infty}\frac{1}{\akvsq}\sum_{j=1}^{n^{k-r_k}}\akjvsq&=&0\label{rkchoose2},
		\end{eqnarray}	
		where
\begin{eqnarray}\label{defax}
\akv&=&\sum_{x \in B_{k}}\ax\\
\akjv&=&\sum_{x \in B_{k,j}}\ax\label{defax2}.
\end{eqnarray}		
\end{defi}
The above definition is directly connected with the Grigelionis Theorem, which is stated below.
		\begin{theorem}\label{grige}(\textbf{Grigelionis} \cite{Grige1}) Let $n_k$, ${k\geq 1}$ be a sequence of natural numbers, let for each $k \geq 1$,  $\xikj$ be independent point processes and let 
		$$\xik^{\omega}=\sum_{j=1}^{n_k}{\xikj^{\omega}}.$$
Assume that  there is a non negative measure $\nu $ such that for all $\BS \subset \R$ 
\begin{eqnarray}
\lim_{k \to{\infty}} \max_{{1}\leq{j}\leq{n_k}}\mathbb{P}\{{\xi}_{k,j}^{\omega}(\BS)\geq{1}\}&=&0,\label{Hip1}\\
				\lim_{k \to{\infty}} \sum_{j=1}^{n_k}{\mathbb{P}\{\xikj^{\omega}(\BS)\geq{1}\}}&=&\nu(\BS),\label{Hip2}\\
				\lim_{k \to{\infty}} \sum_{j=1}^{n_k}{\mathbb{P}\{\xikj^{\omega}(\BS)\geq{2}\}}&=&0\label{Hip3}
\end{eqnarray}
then $\xik^{\omega}$ converges to a Poisson point process with intensity $\nu$.
				\end{theorem}
The following proposition establishes sufficient conditions for the convergence to a Poisson point process in terms of the Hypothesis $(H)$. 
\begin{prop}\label{RESCALE}
Suppose that there is a sequence $\ax$, $x \in \X$ satisfying the hypothesis (H) such that for all bounded measurable set $\BS \in \cB(\R)$
\begin{eqnarray}
E[\xikj^{\omega,e}(I)]\leq \frac{\akj}{\ak}\label{EReScaled},
\end{eqnarray} 
suppose also that for all $e \in \R$ the following regularity condition is fulfilled
\begin{equation}
\lim_{k\lra \infty}\frac{1}{\akv}\sum_{x \in B_k}\denx(e)= F(e)\label{Inten}
\end{equation}
then  $\xik^\omega$ converges to a Poisson process with intensity $F(e)$.
\end{prop}
 \begin{proof}
 We proceed by showing that the conditions of the Grigelionis Theorem \ref{grige} are satisfied. First of all, to prove \ref{Hip1} we use the Chebyshev inequality
\begin{eqnarray}\label{Tcheby}
\mathbb{P}\{\xikjr^{\omega}(\BS)\geq{1}\}&\leq& E[\xikjr^{\omega}(\BS) ] \leq \frac{\akjv}{\akv}\LEBM,
\end{eqnarray} 	
here $|I|$ is the Lebesgue measure ot the bounded Borel set $I$. Therefore
				\begin{eqnarray}
				0\leq \lim_{k \to{\infty}} \max_{{1}\leq{j}\leq n^{k-r_k}}\mathbb{P}\{\xikjr^{\omega}(\BS)\geq{1}\}&\leq&\LEBM\lim_{k \to{\infty}}
\max_{{1}\leq{j}\leq n^{k-r_k}}	\frac{\akjv}{\akv}	
=0.		
				\end{eqnarray}
By using the same argument, we get
\begin{eqnarray}
\lim_{k \to{\infty}} \sum_{j=1}^{n^{k-r_k}}{\mathbb{P}\{\xikjr^{\omega}(\BS)\geq{1}\}} &\leq & \frac{1}{\akv}\sum_{j=1}^{n^{k-r_k}}\akjv\\
&=&1.
\end{eqnarray}
From the dominated convergence theorem and the regularity condition \ref{Inten},  we can deduce  that the limit is $F(e)\LEBM$. Finally, to prove \ref{Hip3} we claim 
\begin{eqnarray}
E[\xikj^{\omega}(\BS)^2]-E[\xikj^{\omega}(\BS)]&\leq&E[\xikjr^\omega(\BS)]^2\label{BOUND2b}.
\end{eqnarray}
Let us suppose that the claim is true. This implies
\begin{eqnarray}
{\mathbb{P}\{\xikj^{\omega}(\BS)\geq{2}\}}&\leq & \sum_{l \geq 2}{\mathbb{P}\{\xikj^{\omega}(\BS)={l}\}}\\
&\leq &\sum_{l \geq 2}l(l-1){\mathbb{P}\{\xikj^{\omega}(\BS)={l}\}}\\
&=&E[\xikj^{\omega}(\BS)^2]-E[\xikj^{\omega}(\BS)]\\
&\leq&E[\xikj^{\omega}(\BS)]^2\\
&\leq & \left(\frac{\akj}{\ake}\LEBM\right)^2
\end{eqnarray}
and by consequence
\begin{equation}
{\lim_{k \to{\infty}} \sum_{j=1}^{n^{k-r_k}}\mathbb{P}\{\xikj^{\omega}(\BS)\geq{2}\}}\leq \LEBM^2\lim_{k \to{\infty}} \sum_{j=1}^{n^{k-r_k}}\frac{\akjesq}{\akesq}=0.
\end{equation}	
It remains to prove  \ref{BOUND2b}.  We get from a direct computation
\begin{eqnarray}
E[\xikj^{\omega}(\BS)^2]&=&E\left[\left(\sum_{x \in B_{k,j}}\delta(\ake(e_{i}^{\omega,k}-e))(\BS)\right)^2\right]\\
&=&E\left[\sum_{x \in B_{k,j}}\delta(\ake(e_{x}^{\omega,k}-e))(\BS)+\sum_{\substack{x  \in B_{k,j},~ y \in B_{k,j}\\x \neq y }}\delta(\ake(e_{x}^{\omega,k}-e))(\BS)\delta(\ake(e_{y}^{\omega,k}-e))(\BS)\right].\nonumber
\end{eqnarray}
The sequence of random variables  $e_{x}^{\omega,k}$, $e_{y}^{\omega,k}$, $x \neq y$ are independent, so
\begin{eqnarray}
E[\xikj^{\omega}(\BS)^2]&\leq &E\left[\sum_{x \in B_{k,j}}\delta(\ake(e_{x}^{\omega,k}-e))(\BS)\right]+E\left[\left(\sum_{x  \in B_{k,j}}\delta(\ake(e_{x}^{\omega,k}-e))(\BS)\right)\right]^2\nonumber,
\end{eqnarray}
which is is equivalent to \ref{BOUND2b}. Since our claim is true, the proof is finished.
\end{proof} 
 The next Theorem  states  a sufficient condition to get the convergence to a Poisson point process in term of  the densities of the random variables $\Vw_x$.  
\begin{coro}\label{Tcheb}
Suppose that  the sequence $||\denx||_\infty$  defined in \ref{densup}, satisfies the Hypothesis $(H)$ and the regularity condition \ref{Inten} is fulfilled, then  $\xik^\omega$ converges to a Poisson point process with intensity $F(e)$.

\end{coro}
\begin{proof}
For all $x \in B_{k,j}$ we have
\begin{eqnarray}
E[\delta(\ake(e_{x}^{\omega,k}-e))(\BS)]
&=&\int_{e+\frac{\BS}{\ake}}\denx(v) dv\\
&\leq&\frac{||\denx||_\infty}{\ake}\LEBM.
\end{eqnarray} 
where $\ak=\sum_{x \in B_{k}}||\denx||_\infty$. By taking the sum over  $x \in B_{k,j}$ and emphasizing that $\akj=\sum_{x \in B_{k,j}}||\denx||_\infty$ we get
\begin{equation}
E[\xikj^\omega(\BS)]\leq \frac{\akj}{\ake}\LEBM,
\end{equation} 
the result is shown by using the same arguments as in the previous proposition.
\end{proof}

\section{Weak Convergence of the Spectral Measure}\label{DOS}
The main objective of this section and the following one is the analysis of the limit behavior of a prescribed sequence of random operators $\Hwk$ of finite range. To do this, in a similar way to the purely random model presented previously, we first study the weak convergence of the spectral measure. To do this, we recall first some useful definitions related to self-adjoint operators and results concerning weak convergence of probability measures.
\subsection{Preliminaries}
\begin{defi}
Given $\omega \in \Omega$, the spectral measure of the operator $\Hw$ at  point $x_0$ is the unique probability measure ${\mu}_{x_0}$ such that for every $f \in C_0(\R)$
\begin{equation}
\int f(u){\mu}_{x_0}(u)=\tx{{x_0}},
\end{equation}
similarly, the expected spectral measure $\avmu $ at $x_0$ is the unique probability measure such that for every $f \in C_0 (\R)$ 
\begin{equation}
\int f(u)d\avmu_{x_0}(u)=\Etx{{x_0}}\nonumber.
\end{equation}
\end{defi}
We emphasize that  the expected spectral measure $\avmu_{x_0}$ depends on the point $x_0$, so its value is  not constant.  However, we can introduce an ''averaged'' spectral measure which is defined more precisely below.
\begin{defi}\label{OpeTr}
	Given $f \in C_0(\R)$ and $\omega \in \Omega$,  the trace of the operator $f(\Hw)$ in  the volume $B_k $ is
	\begin{equation}
	\trB=\sum_{x \in B_k} \tx{x},
	\end{equation}
	similarly, we define the expected trace
	\begin{equation}
	\EtrB=\sum_{x \in B_k} E[\tx{x}].
	\end{equation}
\end{defi}
\begin{defi}
The sequence of probability measures $\{\Pk\}_{k \geq 1}$  converges  weakly to $\P$ if for each $f \in C_0(\R)$ we have
\begin{equation}\label{WConv}
\lim_{k \rightarrow \infty}\int f(v)d\Pk(v)= \int f(v)d\P(v).
\end{equation}
\end{defi}
\begin{theorem}\label{EQUIV}
The following statements are equivalent
\begin{itemize}
\item[i.] The sequence of probability measures  $\{\Pk\}_{k \geq 1}$   converges weakly to $\P$.
\item[ii.] For all $\BS \in \cB(\R)$ we have
\begin{equation}\label{ChConv}
\lim_{k \rightarrow \infty}\int_\BS d\Pk(v)= \int_\BS d\P(v).
\end{equation}
\item[iii.] For all $t \in \R$ we have
\begin{equation}\label{ChConv2}
\lim_{k \rightarrow \infty}\int_\R e^{itv}d\Pk(v)= \int_\R e^{itv}d\P(v).
\end{equation}
\item[iv.] For all $z \in \mathbb{C} ^{+}$ we have
\begin{equation}\label{ChConv3}
\lim_{k \rightarrow \infty}\int_\R \frac{1}{z-v}d\Pk(v)= \int_\R \frac{1}{z-v}d\P(v).
\end{equation}
\end{itemize}
\end{theorem}
The equivalence of the different notions of convergence stated in Theorem \ref{EQUIV} is a widely used result (see for instance \cite{billing} for more details). 
\subsection{Main Theorems}
Given  the sequence of hierarchical balls $B_k (x_0, k) $, $ k \geq 0$, for $k \geq 1$ we take the sequence of random operators
\begin{equation}
\Hwk=\sum_{s=1}^k p_sE_s+\Vw.
\end{equation}
Note that  the  sub-spaces
\begin{equation}
\ell^2(B_k)=\{ \psi \in \ell^2(\X):\psi(x)=0~x \in \X \backslash B_k\}
\end{equation}
are invariant under $\Hwk$. In this case, the spectral counting  measure  is
\begin{equation}\label{COUNTER}
\mu_k^\omega=\sum_{j=1}^{|B_k|}\delta(\ewk{j})
\end{equation}
where $ \ewk{1}<\ewk{2}<\cdots\ewk{{|B_k|}}$ are the eigenvalues of $\Hwk$ restricted to the sub-space $\ell^2(B_k)$, which are isolated since the random variables $\Vw_x$ have strictly positive density. On the other hand,  the  subspace $\ell^2(B_k)$ is invariant under $\Hwk$, so this random operator can be written as a sequence of independent  copies of the operator
\begin{equation}\label{OperRest}
\widetilde{H}^\omega_{B_k}=\sum_{s=1}^{k}p_sE_s+\Vwk
\end{equation}
where $\Vwk$ is a diagonal matrix taking  the value  $V_x$ if  $x \in B_k$ and zero otherwise. The following theorem establishes that the sequence of random $\Hwk$ converges weakly to $\Hw$.

\begin{theorem}\label{WeakConv}
Suppose that for all $t >0$ the density of the random variable $\Vw_x$ satisfies
\begin{equation}\label{KeyCondi}
\KeyCond
\end{equation}
for some sequence $a_x$ satisfying  the Hypothesis $(H)$ and the hierarchical Laplacian fulfils
\begin{equation}\label{condLap}
\lim_{k \rightarrow \infty} \frac{B_k}{\akv}\LArk=0,
\end{equation}
then
\begin{equation}\label{WeakCo}
\lim_{k \rightarrow \infty}E\left[\left(\frac{1}{\COEF}\left(\trB-\EtrB\right)\right)^2\right]=0
\end{equation}

\end{theorem}

\begin{proof}
In a similar way as \cite{Krit1}, we will prove that for all $z \in \mathbb{C}^+$ the complex valued random variable
\begin{equation}\label{Dkz1}
D_{k,\omega}(z)=\int \frac{d\mu_k^\omega(v)}{v-z}-\int \frac{d\avmu_k(v)}{v-z}
\end{equation}
converges almost surely to zero. First of all, we write  for $r_k< k$
\begin{eqnarray}\label{Dkz2}
D_{k,\omega}(z)&=&
\Dkw{1}(z)+\Dkw{2}(z)+\Dkw{3}(z),
\end{eqnarray}
where
\begin{eqnarray}
\Dkw{1}(z)&=&\frac{1}{\COEF}\left\{ \trzBk-\trzBrk\right\}\label{DKW1}\\
\Dkw{2}(z)&=&\frac{1}{\COEF}\left\{ \trBrk-E[\trzBrk]\right\}\label{DKW2}\\
\Dkw{3}(z)&=&\frac{1}{\COEF}\left\{ E[\trzBrk]-E[\trzB]\right\}\label{DKW3}
\end{eqnarray}
here $f_z(v)=(v-z)^{-1}$ and $r_k$, $k \geq 1$ is the same sub-sequence of the Hypothesis $(H)$ . In order to get upper bounds for $\Dkw{1}(z)$ and $\Dkw{3}(z)$, we use the resolvent formula
\begin{equation}
(\Hw_{s-1}-z)^{-1}-(\Hw_s-z)^{-1}=p_s(\Hw_{s-1}-z)^{-1}E_s(\Hw_s-z)^{-1}.
\end{equation}
Taking the sum over $r_k<s\leq k$
\begin{equation}
(\Hwrk-z)^{-1}-(\Hwk-z)^{-1}=-\sum_{s=r_k+1}^{k}p_s(\Hw_{s-1}-z)^{-1}E_s(\Hw_s-z)^{-1},
\end{equation}
thereby
\begin{equation}\label{desresol2}
||(\Hwrk-z)^{-1}-(\Hwk-z)^{-1}||\leq\frac{1}{|\im(z)|^2}\sum_{s=r_k+1}^{k}p_s.
\end{equation}
We emphasize that the previous inequality is obtained since for all  $s \geq 1$  
$$||\Hw_s-z||\leq \frac{1}{|\im(z)|} ,\hspace{1cm} ||E_s||=1.$$ 
The equation \ref{desresol2} directly implies
\begin{eqnarray}\label{Bound1}
|\Dkw{1}(z)|&\leq& \frac{1}{\akv}\sum_{x \in B_k}\sum_{s=r_k+1}^{k}\frac{1}{|\im(z)|^2}p_s\\
&=&	\frac{1}{|\im(z)|^2}\left( \frac{|B_k|}{\akv}\LArk\right).
\end{eqnarray}
Using the same argument, it follows
\begin{eqnarray}\label{Bound3}
|\Dkw{3}(z)|&\leq& 	\frac{1}{|\im(z)|^2}\left( \frac{|B_k|}{\akv}\LArk\right)
\end{eqnarray}
and consequently
\begin{equation}
 \max\{ |\Dkw{1}(z)|, |\Dkw{3}(z)|\}  \leq \frac{1}{|\im(z)|^2} \left( \frac{|B_k|}{\akv}\LArk\right).
\end{equation}
We will get  an estimate for  $E[(\Dkw{2}(z))^2]$. First of all, we set
$$\Uxo=\im (\UHz{x}).$$
Since $B_k$ is the union of $n^{k-r_k}$ disjoint balls with volume $|B_{r_k}|$, the random variable $\Dkw{2}(z)$ can be written  as 
\begin{equation}
\Dkw{2}(z)=\frac{1}{\akv}\sum_{j=1}^{n^{k-r_k}}\Zkj
\end{equation}
where $\Zkj$, $1 \leq j\leq n^{k-r_k}$ is a collection of independent random variables  defined as
\begin{equation}\label{Defkj}
\Zkj=\sum_{x \in B_{k,j}}\Uxo .
\end{equation} 
 From the independence of $\Zkj$ we have
\begin{equation}
E[(\Dkw{2}(z))^2]=\frac{1}{\akv^2}\sum_{j=1}^{n^{k-r_k}}E[\Zkj^2].
\end{equation}
The following arguments are inspired i in the original proof for lemma 2 of Minami \cite{Minami1}. We observe first 
\begin{eqnarray}\label{Inep2}
E[\Zkj^{2}]&=& \sum_{x \in B_{k,j}}\sum_{y \in B_{k,j}}E[\Uxo \Uyo]\\
&=& \sum_{x \in B_{k,j}}\sum_{y \in B_{k,j}}  E \det \begin{pmatrix}
\im\left\langle\delta_{x} , (\Hwrk-z)^{-1}\delta_{x}\right\rangle & \im\left\langle\delta_{x} , (\Hwrk-z)^{-1}\delta_{y}\right\rangle \\
\im\left\langle\delta_{y} , (\Hwrk-z)^{-1}\delta_{x}\right\rangle	& \im\left\langle\delta_{y} , (\Hwrk-z)^{-1}\delta_{y}\right\rangle
\end{pmatrix} .\nonumber
\end{eqnarray}
By calling $G_{\Hwrk}(x,y;z)=\langle\delta_{x} , (\Hwrk-z)^{-1}\delta_{y}\rangle$ from the Krein's  formula we know that for $z \in \C^+$
\begin{equation}
\begin{pmatrix}
G_{\Hwrk}(x,x;z) & G_{\Hwrk}(x,y;z)\\
G_{\Hwrk}(y,x;z)	& G_{\Hwrk}(y,y;z)
\end{pmatrix}  =\left(\begin{pmatrix}
\Vw_x
&0\\0&\Vw_y
\end{pmatrix}+
\begin{pmatrix}
\SG_{{\tHwrk}}(x,x;z) & \SG_{ \tHwrk}(x,y;z)\\
\SG_{ \tHwrk}(y,x;z)	& \SG_{ \tHwrk}(y,y;z)
\end{pmatrix}^{-1} \right)^{-1} \nonumber
\end{equation}
where ${\tHwrk}=\Hwrk-V_x \delta_x -V_y \delta_y$ do not depend on $\Vw_x,\Vw_y$. To simplify notation we write
\begin{equation}
V(\Vw_x,\Vw_y)=\begin{pmatrix}
\Vw_x
&0\\0&\Vw_y
\end{pmatrix}\hspace{1cm}
\tilde{G}(z)=\begin{pmatrix}
\SG_{{\tHwrk}}(x,x;z) & \SG_{ \tHwrk}(x,y;z)\\
\SG_{ \tHwrk}(y,x;z)	& \SG_{ \tHwrk}(y,y;z)
\end{pmatrix}^{}.
\end{equation}
From a direct computation we deduce
\begin{equation}\label{CopyMin}
\Uxo \Uyo= \frac{\det \im(\tilde{G}(z))}{|\det \tilde{G}(z)|^2}\frac{1}{|V(\Vw_x,\Vw_y)-\tilde{G}(z)^{-1})|^2}.
\end{equation}
It is a well known fact that 
$$E[\Uxo \Uyo]=E[E[\Uxo \Uyo|\cF_{x,y}]]$$
where $\cF_{x,y}$ is the sigma algebra generated by all the random variables excluding  $V_x,V_y$. By using  equation \ref{CopyMin} we can write 
\begin{equation}\label{CopyMin}
E[\Uxo \Uyo|\cF_{x,y}]= \frac{\det \im(\tilde{G}(z))}{|\det \tilde{G}(z)|^2} \int_\R \int_\R \frac{\denx(u)\deny(v)}{|\det(V(u,v)-\tilde{G}(z)^{-1})|^2}dv du.
\end{equation}
We claim that under assumption \ref{KeyCondi}
\begin{equation}\label{ClaimMin}
E[\Uxo \Uyo|\cF_{x,y}]\leq a_x a_y.
\end{equation}
If the claim is true, we have $E[\Zkj^{2}]\leq \akjvsq$ and consequently
\begin{equation}\label{Bound2}
E[(\Dkw{2}(z))^2]\leq \frac{1}{\akvsq}\sum_{j=1}^{n^{k-r_k}}\akjvsq.
\end{equation}
 To prove the claim, we first set
 \begin{equation}
\tilde{G}(z)=\begin{pmatrix}
a_1&b_1\\b_1&c_1
\end{pmatrix}+i\begin{pmatrix}
a_2&b_2\\b_2&c_2
\end{pmatrix}
 \end{equation}
(here  all the coefficients $a_i,b_i,c_i$ $i=1,2$ are real numbers). This notation allows us to write the integral at the right hand of \ref{CopyMin} as
\begin{equation}\label{InteMin} 
I_{x,y}= \int_\R \frac{\denx(u)}{|u+ia_2|^2}\int_\R \frac{\deny(v)}{\left|v+\left( ic_2-\frac{(b_1+ib_2)^2}{u+ia_2}\right)\right|^2}dv du.
\end{equation}
We use twice  the assumption \ref{KeyCondi} to get
\begin{eqnarray}\label{InteMin2} 
I_{x,y}&\leq& a_y \int_\R \frac{\denx(u)}{|u+ia_2|^2} \frac{1}{\im \left( ic_2-\frac{(b_1+ib_2)^2}{u+ia_2}\right)} du\\
&=&a_y \frac{1}{c_2}\int_\R \frac{\denx(u)}{ \left( u-\frac{b_1b_2}{c_2}\right)^2+\frac{1}{c_2^2}\Delta (b_1^2+b_2^2+\Delta)} du\\
&\leq& a_x a_y\frac{1}{\sqrt{\Delta (b_1^2+b_2^2+\Delta)}}\\
&\leq&\frac{a_x a_y}{\Delta},
\end{eqnarray}
 where $\Delta=|a_2c_2-b_2^2|>0$.  To obtain \ref{ClaimMin} we observe that
$\Delta \geq \frac{|\det \im(\tilde{G}(z))|}{|\det \tilde{G}(z)|^2}$.
Finally, from  inequalities \ref{Bound1}, \ref{Bound3} and \ref{Bound2} we get
\begin{equation}\label{KeyIneq}
E[|D_{k,\omega}|^2]\leq \left(\frac{2}{|\im(z)|^2}  \frac{|B_k|}{\akv}\LArk\right)^2+\frac{1}{\akvsq}\sum_{j=1}^{n^{k-r_k}}\akjvsq
\end{equation}
concluding the desired result from  assumptions \ref{condLap} and \ref{rkchoose2} of Hypothesis $(H)$ and letting $k \lra \infty$
\end{proof}
The following theorem states that under some additional conditions,  we also get the almost sure convergence.
\begin{theorem}\label{StrongConv}
	Under assumptions of Theorem \ref{WeakConv}, if
	\begin{eqnarray}
	\sum_{k \geq 1} \frac{1}{\akesq}\sum_{k=1}^{n^{k-r_k}}\akjesq&<&\infty\\
	\sum_{k \geq 1} \left(\frac{B_k}{\ak}  \LArk\right)^p&<&\infty \hspace{1cm}\textrm{for some $p > 1$} \label{summable},
	\end{eqnarray}
	then $$P\left[\omega \in \Omega :\lim_{k\rightarrow  \infty}\left|\frac{1}{\COEF}\left(\trB-\EtrB\right)\right|=0\right]=1.$$
\end{theorem}
\begin{proof}
	We will prove that for all $\delta>0$ the event
	\begin{equation}
	\Omega_\delta=\left\{ \omega \in \Omega : \limsup |D_{k,\omega}|<\delta \right\}
	\end{equation}
	has probability one. From the obvious relation $P[\Omega_\delta]=1-P[\Omega_\delta^c]$, we only have to show that $P[\Omega_\delta^c]=0$, where
		\begin{equation}
	\Omega_\delta^c=\left\{ \omega \in \Omega : \bigcap_{K \geq 1}\bigcup_{k\geq K} |D_{k,\omega}|\geq \delta \right\}.
	\end{equation}
	It is a well known fact that
	\begin{eqnarray}
	P[\Omega_\delta^c]&=&\lim_{K\rightarrow \infty}P\left[\bigcup_{k\geq K} |D_{k,\omega}|\geq \delta\right]\\
	&\leq &\lim_{K\rightarrow \infty}\sum_{k \geq K}P\left[ |D_{k,\omega}|\geq\delta\right].
	\end{eqnarray}
From the Chebyshev's inequality, we have for all $\delta>0$
	\begin{eqnarray*}
P\left[\bigcup_{k\geq K} |D_{k,\omega}(z)|\geq \delta\right]&\leq& P\left[\bigcup_{k\geq K} |\Dkw{1}(z)+\Dkw{3}(z)|\geq \frac{\delta}{2}\right]+P\left[\bigcup_{k\geq K} |\Dkw{2}(z)|\geq \frac{\delta}{2}\right]\\
&\leq &\sum_{k \geq K}\frac{E[|\Dkw{1}(z)+\Dkw{3}(z)|^p]}{(\delta/2)^p}+\sum_{k \geq K}\frac{E[D_{k,\omega}(z)^2]}{(\delta/2)^2}\\
&\leq& \frac{1}{(\delta/2)^{\min\{2,p\}}}\sum_{k \geq K}\left[\left(\frac{2}{|\im(z)|^2}  \frac{|B_k|}{\akv}\LArk\right)^p+\frac{1}{\akvsq}\sum_{j=1}^{n^{k-r_k}}\akjvsq\right].
	\end{eqnarray*}
	If assumption \ref{summable} is fulfilled, then $\lim_{K\rightarrow \infty}\sum_{k \geq K}P\left[ |D_{k,\omega}(z)|\geq\delta\right]=0$ and consequently $P[\Omega_\delta]=1$. The proof concludes by taking the event $\tilde{\Omega}=\bigcap_{m \geq 1} \Omega_{1/m}$.
		\end{proof}
	
The last theorem of this section states that  the spectral measurement is absolutely continuous with respect to the Lebesgue measure
\begin{theorem}\label{ABSCONT}
Under the main assumptions of Theorem  \ref{WeakConv}, the expected spectral measure $\avmu$ is absolutely continuous respect to the Lebesgue measure, i.e. there is a density function  $\eta$ such that for all  $\BS \in \cB(\R)$
\begin{equation}
\avmu(\BS) = \int_\BS \eta(e)de
\end{equation}
and moreover $||\eta||_\infty<\infty$.
\end{theorem}
\begin{proof}
We have to prove that   there is a density function $\eta$ such that all $f \in C_0(\R)$
\begin{equation}\label{muav}
\int_\R f(e) d\avmu(e)=\lim_{k \rightarrow \infty}\frac{1}{\COEF}\trB.
\end{equation}
Since the limit exists, we only have to prove that the limiting measure is absolutely continuous, which is equivalent to prove that for all $z \in \mathbb{C}^{+}$ 
\begin{equation}\label{AbsCon}
\frac{1}{\pi}\im \int \frac{\avmu(dv)}{v-z}\leq 1
\end{equation}
(see for instance lemma 4.2 of \cite{Aize1}). From the one-rank perturbation formula, we have for all $x \in \X$ and $f(t)=(t-z)^{-1}$, $z \in \C^+$
\begin{equation}
\im \tx{x} =\im \frac{1}{V_x-\Sigma_0(z)},
\end{equation}
where $\Sigma_0(z)=\txo{x}$, $\Hw_x=\Hw-V_x\delta_x$. By taking conditional expectation we get
\begin{eqnarray}
E\left[\im \frac{1}{V_x-\Sigma_0(z)}\right]&=&E\left[E\left[\im \frac{1}{V_x-\Sigma_0(z)}|\cF_x\right]\right]
\end{eqnarray}
where $\cF_x$ is the sigma algebra generated by all the random variables excluding  $\Vw_x$. From  the assumption \ref{KeyCondi} we get
\begin{eqnarray}
E\left[\im \frac{1}{V_x-\Sigma_0(z)}|\cF_x\right]&=&\int_{\R}\frac{\im \Sigma_0(z)}{(u-\re \Sigma_0(z))^2+\im \Sigma_0(z)^2}\denx(u)du\\
&\leq&\ax.
\end{eqnarray}
We take the sum over $x \in B_k$ to obtain for all $k \geq 1$
\begin{equation}\label{muavb}
\frac{1}{\COEF}\EtrB\leq 1,
\end{equation}
concluding the proof by letting $k \rightarrow \infty$.
\end{proof}

\section{ Poisson Statistics}\label{POISSTA}
\begin{theorem}\label{Poissonpp}
Suppose that density of the random variable $\Vw_x$ satisfies  
\begin{equation}\label{MomCond}
\KeyCond
\end{equation}
for some sequence $\ax$, $x \in \X$ satisfying the hypothesis $(H)$  and the hierarchical Laplacian satisfies
 \begin{equation}\label{LapCondi}
\lim_{k \rightarrow \infty} {\akv |B_{k}|}\LArk=0.
\end{equation} 
We  also assume
\begin{equation}\label{CONDREG2}
\lim_{\varepsilon\rightarrow 0^+}\im \int_\R\frac{1}{v-e-i\varepsilon }\eta(v)dv=\pi \eta(e),
\end{equation}
where $\eta$ is the density of the expected spectral measure \ref{muav}. If the above conditions are fulfilled, then the  rescaled eigenvalue counting measure 
$$ \xik^\omega=\sum_{i=1}^{|B_k|}\delta(\ak(e_k^{\omega,e}-e))$$
 converges to a Poisson point process with intensity  $\eta(e)$.
\end{theorem}
\begin{proof}
In the same way as \cite{Krit1}  we   first approximate  the eigenvalue count process in a convenient way. Let us define
\begin{eqnarray}
{\txik}^{\omega,e}&=&\sum_{j=1}^{n^{k-r_k}} \txikj^{\omega,e}
\end{eqnarray}
where $\txikj$, $1 \leq j \leq n^{k-r_k}$ is the eigenvalue counting measure  of  the operator  $\Hwrk$ restricted to the subspace  $\ell^2(B_{k,j})$ (already defined in equation \ref{OperRest}). From the construction of the operator we know that $\txikj$ is independent of $\txikjb$ when $j \neq j^\prime$.  As a first step of the proof, we state next that   $\xik$  and $\txik$ are asymptotically equivalent in the weak sense.
\begin{prop}\label{propL1}
For all $f \in L^1(\R)$ we have
\begin{equation}\label{limitf}
\lim_{k \rightarrow \infty}E\left[\left|\int_\R f d\xik-\int_\R f d\txik\right|\right]=0.
\end{equation}
\end{prop}
\begin{proof}
From Theorem \ref{EQUIV}, it suffices to show that the limit \ref{limitf} applies for the family of functions  on the form $f_z(v)=\im(v-z)^{-1}$, $z \in \mathbb{C}^+$. By setting  $z_k=\re(z)+\frac{\im(z)}{\ak}i$ we get
\begin{equation}
\int_\R f_{z_k} d\xik-\int_\R f_{z_k} d\txik=
\frac{1}{\akv}\sum_{x \in B_{k,j}} \EHxxk{x}-\EHxxbk{x}\label{DesH2}.
\end{equation}
The resolvent formula (see equation \ref{desresol2}) allows us to write
\begin{equation}\label{desresol3}
||(\Hw_{k}-z_k)^{-1}-(\Hwrk-z_k)^{-1}||\leq\frac{1}{|\im(z_k)|^2}\LArk
\end{equation}
and therefore
\begin{equation}
E\left[\left|\int_\R f_{z_k} d\xik-\int_\R f_{z_k} d\txik\right|\right]\leq {\akv |B_{k}|}\LArk.
\end{equation}
We deduce the result from \ref{LapCondi} and letting $k \rightarrow \infty$.
\end{proof}
We now return to the proof of the main result: from Proposition \ref{propL1} and  the Grigelionis Theorem we know that  it is enough to prove 
				\begin{eqnarray}
				\lim_{k \to{\infty}} \max_{{1}\leq{j}\leq{n^{k-r_k}}}E[\txikj(\BS)]&=&0,\label{Hip1b}\\
				\lim_{k \to{\infty}} \sum_{j=1}^{n^{k-r_k}}{\mathbb{P}\{\txikj^{\omega}(\BS)\geq{2}\}}&=&0\label{Hip3b}\\
\lim_{k \to{\infty}} \sum_{j=1}^{n^{k-r_k}}E[\txikj(\BS)]&=&\eta(e)\label{Hip4b}.
				\end{eqnarray}
To prove these conditions, we first recall that for all bounded interval  $I=\left[s-\frac{t}{2},s+\frac{t}{2}\right]$ and $x \in \X$  the next inequality is satisfied
\begin{equation}\label{REPLACE}
\PIx \leq |t|\im \RIx,
\end{equation}
where $P_I(\Hw_{k})=\mathds{1}_I(\Hwk)$. In particular, for $I=\left[e-\frac{|I|}{2\ak},e+\frac{|I|}{2\ak}\right]$
\begin{eqnarray}
E[\txikj^{\omega,e}(\BS)]
&\leq&\frac{|I|}{\ak}\sum_{x \in B_{k,j}}\EHxkz{x}\\
&\leq&\frac{|I|}{\ak}\sum_{x \in B_{k,j}}\ax\\
&=&|I|\frac{\akj}{\ak},
\end{eqnarray}
therefore
\begin{equation}\label{DesH}
 \limsup_{k \rightarrow \infty}  \max_{{1}\leq{j}\leq{n^{k-r_k}}} E[\txikj^{\omega,e}(\BS)]\leq  \lim _{k \rightarrow \infty} \max_{1\leq j \leq{n^{k-r_k}}}\frac{\akj}{\ak}=0.
\end{equation}
To prove  \ref{Hip3b} we  set
\begin{equation}
\Zkjzk=\sum_{k \in B_{k,j}}\Hxkz{x},
\end{equation}
here $z_k=e+\frac{|I|}{2\ak}$. From   \ref{REPLACE}  and the Chebyshev's inequality we derive 
\begin{eqnarray}
\mathbb{P}\left[\txikj^{\omega}(\BS)\geq{2}\right]&\leq& \mathbb{P}\left[\Zkjzk\geq{2}\right]\\
&=&\mathbb{P}\left[\Zkjzk-E[\Zkjzk]\geq{2}-\Zkjzk\right]\\
&\leq& \frac{E\left[\left(\Zkjzk-E[\Zkjzk]\right)^2\right]}{({2}-E[\Zkjzk])^2}.
\end{eqnarray}
For $k$ large enough we know that $E[\Zkjzk]\leq |I|\frac{\akj}{\ak} \leq 1$,
hence
\begin{eqnarray}
\mathbb{P}\left[\txikj^{\omega}(\BS)\geq{2}\right]&\leq& E\left[\left(\Zkjzk-E[\Zkjzk]\right)^2\right]\\
&\leq&\left(|I|\frac{\akj}{\ak}\right)^2,
\end{eqnarray}
where the last inequality can be deduced from the identity
\begin{eqnarray*}
{E\left[\left(\Zkjzk-E[\Zkjzk]\right)^2\right]}&=& 
\sum_{x \in B_{k,j}}\sum_{y \in B_{k,j}}  E \det \begin{pmatrix}
\im\left\langle\delta_{x} , (\Hwrk-z_k)^{-1}\delta_{x}\right\rangle & \im\left\langle\delta_{x} , (\Hwrk-z_k)^{-1}\delta_{y}\right\rangle \\
\im\left\langle\delta_{y} , (\Hwrk-z_k)^{-1}\delta_{x}\right\rangle	& \im\left\langle\delta_{y} , (\Hwrk-z_k)^{-1}\delta_{y}\right\rangle
\end{pmatrix}
\end{eqnarray*}
and emphasizing that the same argument used to estimate $\Dkw{2}$ follows. This yield to
\begin{equation}
 \lim_{k \to{\infty}} \sum_{j=1}^{n^{k-r_k}}{\mathbb{P}\{\txikj^{\omega}(\BS)\geq{2}\}}\leq |I|^2\lim_{k\rightarrow \infty}\frac{1}{\akesq}\sum_{j=1}^{n^{k-r_k}}\akjesq=0.
\end{equation}

Finally, to prove \ref{Hip4b},  we write
\begin{eqnarray}
\sum_{j=1}^{n^{k-r_k}}E[\im(e-z_k)^{-1}\xikj^{\omega,e}(v)]
&=&\frac{1}{\akv}\sum_{x \in B_{k}} \EHxkz{x}- \EHxbz{x}\nonumber\\
&+&\frac{1}{\akv}\sum_{x \in B_{k}} \EHxbz{x}\label{DesH2b}.
\end{eqnarray}
From the resolvent formula (see \ref{desresol2}) we get
\begin{equation}
\left|\frac{1}{\akv}\sum_{x \in B_{k}} \EHxkz{x}- \EHxbz{x}\right|\leq  \akv |B_{k}|\LArk,
\end{equation}
whose limit  is zero. Finally,   the regularity condition  \ref{CONDREG2} directly implies
\begin{eqnarray}
\lim_{k\rightarrow \infty}\frac{1}{\akv}\sum_{x \in B_{k}} \EHxbz{x}=\pi \eta(e),
\end{eqnarray}
concluding the proof.
\end{proof}
Concerning the main assumptions of this work, we wake bellow some comments:
\begin{itemize}
	\item If the sequence  $||\denx||_\infty$ satisfy  the hypothesis $(H)$, assumption \ref{KeyCondi} follows directly since for all $s \in \R$
	\begin{eqnarray}
	\KeyCondC\\&\leq&{\pi}||\denx||_\infty.
	\end{eqnarray}
	This is implicitly derived in the well known inequalities of \cite{Wegner1} and \cite{Minami1}, commonly used to prove localization of eigenvalues.
	\item The assumption on the hierarchical Laplacian \ref{LapCondi} is connected with the spectral dimension of the operator and can be interpreted as a competition between the deterministic and random components of the operator (the latter implicitly represented by the value $\ak$).
	\item The regularity condition in general is met when the random variables  $V_x$, $x \in \X$ are smooth enough (we have omitted more details about the regularity of the density since our main concern is the non-homogeneity of the random field).
\end{itemize}

\section{The Hierarchical Anderson Model with non i.i.d random potential}\label{EXAMPLES}
\newcommand{\oone}{\left(1+o_k(1)\right)}
\newcommand{\oneb}{o_k(\cdot)}
We are particularly interested into the  case $$C_1(1+|x|)^{\pot} \leq \KeyInt \leq C_2(1+|x|)^{\pot}$$
 where $\gamma>-1$ and  $0\leq C_1 \leq C_2<\infty$. This example is important because has a connection with the study of spin systems with  non homogeneous external random fields: in \cite{Bissacot2015} \cite{Bissacot2010}  the two-dimensional Ising Model was considered, whereas in the most recent work \cite{Bissacot2017} similar results for the one-dimensional Dyson Model can be founded.
 Concerning our example, the natural choice is the sequence
$$\akjv=(1+\pot) \sum_{x \in B_{k,j}} (1+|x|)^{\pot}$$
with the convention $A_{k,1}=A_{k-1}$.  We state next the two main Theorems related to this  example.
\begin{theorem}\label{TeoApl}
	For all $0<d<\frac{1}{1+\frac{\pot}{2}}$ we have
	\begin{equation}\label{WeakCoApl}
P\left[	\lim_{k \rightarrow \infty}\frac{1}{\COEF}\left(\trB-\EtrB\right)=0 \right]=1. 
	\end{equation}
Moreover, there is a density $\eta$ such that for all $f \in C_0(\R)$
	\begin{equation}\label{muav}
	\int_\R f(e) d\avmu(e)=\lim_{k \rightarrow \infty}\frac{1}{\COEF}\trB.
	\end{equation}
\end{theorem}
\begin{theorem}\label{PoissonppAlp}
	For all $0<d<\frac{1}{1+\frac{\pot}{2}}$, if  the re-scaled counting measure  $\xik^\omega$ satisfies the regularity condition
	\begin{equation}\label{CONDREG2}
	\lim_{\varepsilon\rightarrow 0^+}\im \int_\R\frac{1}{v-e-i\varepsilon }\eta(v)dv=\pi \eta(e),
	\end{equation}
 then $\xik^\omega$	converges to a Poisson point process with intensity  $\eta(e)$.
\end{theorem}
\subsection{Proof of Theorem \ref{TeoApl}}
The  proof is based on checking  that the assumptions of the main theorems stated in  the previous sections are met. To do this, we first  approximate these sums by integrals, obtaining   for $k$ large enough
\begin{eqnarray}
\akv&=&{|B_k|^{1+\pot}}\oone\\
A_{k,1}&=&{|B_{k,1}|^{1+\pot}}\oone\label{ak1}\\
\akjv&=& \sum_{x=1+jn^{k-r_k}}^{(j+1){n^{k-r_k}}} (1+|x|)^{\pot}  = {{j^\pot|B_{k,j}|^{\pot}}}\oone,\hspace{1cm}2 \leq j \leq \frac{|B_k|}{|B_{k,j}|}\label{akj}.
\end{eqnarray}
From equation \ref{akj} we have for $k$ large enough and $\pot \neq -1/2$
\begin{eqnarray}
\sum_{j=1}^{n^{k-r_k}}\akjesq&=& \left(|B_{k,j}|^{2+2\pot}+|B_{k,j}|^{2\pot}\sum_{j=2}^{n^{k-r_k}}j^{2\pot} \right)\oone\\
&\leq&  \left(|B_{k,j}|^{2+2\pot}+\frac{1}{|1+2\pot|}|B_{k,j}|^{2\pot}\left(\frac{|B_k|}{|B_{k,j}|}\right)^{1+2\pot}\right)\oone\label{aproxakj}.
\end{eqnarray}
Similarly, we have for $\pot=-1/2$
\begin{eqnarray}\label{Special}
\sum_{j=1}^{n^{k-r_k}}\akjesq&\leq&  \left(|B_{k,j}|^{2+2\pot}+|B_{k,j}|^{2\pot}\ln\left(\frac{|B_k|}{|B_{k,j}|}\right)\right)\oone\label{aproxakjb}.
\end{eqnarray}
In the following we will omit the case $\pot=-\frac{1}{2}$ (nevertheless,  we emphasize that the same arguments are valid). The next proposition that the sequence fulfils the Hypothesis $(H)$.
\begin{prop}\label{HipH}
	The sequence $\ax=(1+|x|)^\pot$ satisfies the Hypothesis $(H)$ for all $\pot >-1$ and moreover
		\begin{eqnarray}
	\sum_{k \geq 1} \frac{1}{\akesq}\sum_{k=1}^{n^{k-r_k}}\akjesq&<&\infty\label{suma}.
	\end{eqnarray}
\end{prop}
\newcommand{\thetare}{\theta_k}
\begin{proof}
We  take the simplest choice: given $\theta \in (0,1)$ fixed, we take $r_k=\floor{\theta k}$, $\thetare=\frac{r_k}{k}$ (here the symbol $\floor{\phantom{\theta}}$ denotes the entire part). This implies  that $\theta-\frac{1}{k}\leq \thetare\leq \theta $ for all $k .\geq 1$. In addition, for  $1 \leq j\leq |B_k|^{1-\thetare}$ we have  $|B_{k,j}|= |B_k|^{\thetare}=|B_k|^{\theta}\oone$.
 We will check below that the selected sequence satisfy the Hypothesis $(H)$.\\
\textbf{Proof of  \ref{rkchoose}:} From the estimates  \ref{ak1}, \ref{akj} we get
	\begin{eqnarray}
\limsup_{k \rightarrow \infty}\sup_{1\leq j\leq n^{k-r_k}}\frac{\akjv}{\akv}&\leq&\limsup_{k \rightarrow \infty}\frac{\max\left\{|B_{k,j}|^{1+\pot},|B_{k}|^{\pot}\right\}}{|B_k|^{1+\pot}}\\
&=&\lim_{k \rightarrow \infty}|B_{k,j}|^{\max\{(\thetare-1)(1+\pot),-1\}}\\
&=&0.
	\end{eqnarray}
	The last inequality is valid since  $(\thetare-1)(1+\pot)\leq (\theta-1)(1+\pot)<0$ for all $k \geq 1$.	\\
\textbf{Proof of \ref{rkchoose2}:} from equation \ref{aproxakj} and recalling that $|B_{k,j}|= |B_k|^{\thetare}$, we get for $\pot \neq -1/2$
		\begin{eqnarray}
	\frac{1}{\akvsq}\sum_{j=1}^{n^{k-r_k}}\akjvsq&=&\frac{1}{\akvsq}\left( |B_{k,j}|^{2+2\pot}+|B_{k,j}|^{2\pot}\left(\frac{|B_k|}{|B_{k,j}|}\right)^{1+2\pot}\right)\oone\\
	&=& \frac{1}{|B_k|^{2+2\pot}}\left( |B_{k}|^{(2+2\pot)\thetare}+|B_{k}|^{2\pot\thetare}\left(\frac{|B_k|}{|B_{k}|^{\thetare}}\right)^{1+2\pot}\right)\oone\\
	&=& (|B_{k}|^{2(1+\pot)(\thetare-1)}+ |B_{k}|^{-(\thetare+1)})\oone,
	\end{eqnarray}
since $0<\lim_{k\lra \infty}\thetare<1$, given $\varepsilon>0$ and  $k\geq k_0$ large enough,
 we have $\theta(1-\varepsilon)<\theta_k\leq \theta<1$. By taking $\theta^{\star}=\min\{2(1+\pot)(1-\theta),(\thetare+1-\varepsilon)\}>0$ we obtain
 		\begin{eqnarray}\label{sumtail}
 \frac{1}{\akvsq}\sum_{j=1}^{n^{k-r_k}}\akjvsq
 &\leq& |B_{k}|^{-\theta^{\star}}
 \end{eqnarray}
 and consequently  $\lim_{k\rightarrow \infty}\frac{1}{\akvsq}\sum_{j=1}^{n^{k-r_k}}\akjvsq=0$. \\
 \textbf{Proof of \ref{suma}:}   we only need to show that the tails of the sum are convergent. From equation \ref{sumtail} we have for all $k \geq K$
  		\begin{eqnarray}\label{sumtailb}
 \sum_{k \geq K}\frac{1}{\akvsq}\sum_{j=1}^{n^{k-r_k}}\akjvsq
 &\leq& \sum_{k \geq K} |B_{k}|^{-\theta^{\star}}\oone\\
 &\leq&\frac{|B_{K}|^{-\theta^{\star}}}{1-n^{-\theta^*}}\oone\\
 &<&\infty.
 \end{eqnarray}
\end{proof}	
The   above proposition  sufficient conditions in order that the hypothesis $(H)$ can be accomplished.  The next proposition, states the connection between the spectral dimension of the Hierarchical  Laplacian and the parameter $\pot$.
\begin{prop}
Let $d$ be the spectral dimension of the hierarchical Laplacian. Then, for all $0<d<\frac{1}{1+\frac{\pot}{2}}$ we have
\begin{eqnarray}\label{keyc}
\lim_{k\rightarrow  \infty}\ak |B_k| \LArk&=&0,\\
\sum_{k \geq 1} \left(\frac{B_k}{\ak}  \LArk\right)^p&<&\infty \hspace{1cm}\textrm{for all $p > 1$} \label{summableb}.
\end{eqnarray}
\end{prop}
\begin{proof}
In a similar way as the proof of proposition \ref{HipH}, if $0<d<1$  is the spectral dimension of the Laplacian, we fix $\theta$ satisfying $(1+\frac{\pot}{2})d<\theta<1$ and we also fix the sequence $r_k=\floor{\theta k}$. From the definition of $d$ we know that  $\sum_{s=r_k+1}^{\infty}p_s= n^{-\frac{2r_k}{d}+\epsilon_k}= |B_k|^{-\frac{2\thetare}{d}+\epsilon_k}$ for some sequence $\epsilon_k\rightarrow 0^+$, so
\begin{eqnarray}
\ak |B_k| \sum_{s=r_k+1}^{\infty}p_s &=& |B_k|^{2+\pot}|B_k|^{-\frac{2\thetare}{d}+\epsilon_k}\oone\\
&\leq& |B_k|^{2+\pot-\frac{2\theta}{d}+\epsilon_k+\frac{1}{k}}\oone\label{HierLa}.
\end{eqnarray}
Now, given $\varepsilon_0>0$ we have for  $K$ large enough that  $|\epsilon_k+\frac{1}{k}| \leq \varepsilon_0|2+\pot-\frac{\theta}{d}|$, then
\begin{eqnarray}
\ak |B_k| \sum_{s=r_k+1}^{\infty}p_s 
&\leq & |B_k|^{(1-\varepsilon_0)\left(2+\pot-\frac{2\theta}{d}\right)}
\end{eqnarray}
 and consequently the inequality \ref{keyc} is obtained by letting $k \lra \infty$ and recalling $2+\pot-\frac{2\theta}{d}<0$. To prove  \ref{summableb}, we use the same estimates to deduce
 \begin{eqnarray}
\frac{ |B_k|}{\ak} \LArk &\leq& |B_k|^{-\pot}|B_k|^{-\frac{2\thetare}{d}+\epsilon_k}\oone\\
 &\leq& |B_k|^{-\pot-\frac{2\theta}{d}+\epsilon_k+\frac{1}{k}}\label{HierLaB}\oone\\
 &\leq& |B_k|^{-\pot-\frac{2\theta}{d}+\varepsilon_0|2+\pot-\frac{\theta}{d}|}\oone\label{HierLaC}.
 \end{eqnarray}
 By noticing that $2+\pot>-\pot$ when $\pot>-1$, we can write
  \begin{eqnarray}
 \frac{ |B_k|}{\ak} \sum_{s=r_k+1}^{\infty}p_s &\leq& |B_k|^{(1-\varepsilon_0)\left(2+\pot-\frac{2\theta}{d}\right)}\label{HierLaD}\oone.
 \end{eqnarray}
 The above inequality allows us to conclude that for all $p>1$ and $K$ large enough
   \begin{eqnarray}
 \sum_{k \geq K}\left(\frac{ |B_k|}{\ak} \LArk\right)^p &\leq&\oone \sum_{k \geq K}|B_k|^{(1-\varepsilon_0)\left(2+\pot-\frac{2\theta}{d}\right)p}\\
 &<&\infty.
 \end{eqnarray}
We conclude the proof emphasizing that for $\pot=-\frac{1}{2}$ the same arguments are valid  if we use  \ref{Special} instead of  \ref{aproxakj}.\\
\end{proof}
\section*{Acknowledgments}
The author acknowledges the financial  support  of the fellowship program FONDECYT de Iniciaci\'on en Investigaci\'on, Project No. 11140479  and to the following Mathematics Institutes where part of this work were done: Universidad Cat\'olica del Norte (Antofagasta, Chile), I2M (Marseille, France), Center of Mathematical Modeling (Santiago, Chile) and Instituto Potosino de Investigaci\'on Cient\'ifica y Tecnol\'ogica (San Luis Potos\'i, M\'exico). The author would like also to  Ariel P\'erez, Andrea Hern\'andez and Mar\'ia Villegas for their helpful contribution  at the beginning of this work, in the context of their corresponding Msc. Thesis. 
\bibliographystyle{apalike}
\bibliography{Bibliography}
\end{document}